\documentclass[12pt,reqno]{article}

\usepackage{amsmath, amssymb, amsthm, graphicx}
\usepackage{fancyhdr}

\newtheorem{thm}{Theorem}[section]
\newtheorem{lem}[thm]{Lemma}

\newtheorem{prop}[thm]{Proposition}
\newtheorem{ques}[thm]{Question}

\begin{document}

\pagenumbering{arabic} \setcounter{page}{1}

\title{Sum of many dilates}

\author{George Shakan\\
University of Wyoming Department of Mathematics \\
Laramie, Wyoming 82072, USA \\
george.shakan@gmail.com}

\maketitle


\begin{abstract}
We show that for any coprime integers $\lambda_1 , \ldots , \lambda_k$ and any finite $A \subset \mathbb{Z}$, one has $$|\lambda_1 \cdot A + \ldots + \lambda_k \cdot A| \geq (|\lambda_1| + \ldots + |\lambda_k|)|A|- C,$$ where $C$ only depends on $\lambda_1 , \ldots , \lambda_k$. 
\end{abstract}
\section{Introduction}

Let $A$ and $B$ be finite sets of real numbers. The sumset of $A$ and $B$ is defined by $$A + B = \{a + b : a \in A, \ b \in B \}.$$ For a real number $d \neq 0$ the dilation of $A$ by $d$ is defined by 
$$d \cdot A = \{d\} \cdot A =  \{da : a \in A\},$$ while for any real number $x$, the translation of $A$ by $x$ is defined by 
$$x + A = \{x\} + A = \{x + a : a \in A\}.$$ Our main theorem is the following. 
\begin{thm}\label{main}
	Let $A$ be a finite subset of the integers and $\lambda_1 , \ldots, \lambda_k$ be coprime integers. Then one has $$|\lambda_1 \cdot A + \ldots + \lambda_k \cdot A| \geq (|\lambda_1| + \ldots + |\lambda_k|) |A| - C,$$ where $C$ only depends on $\lambda_1 , \ldots , \lambda_k$. 
\end{thm} 

We show that one may take $C =\frac{1}{3} {k+1 \choose 2} |\lambda_1 \cdots \lambda_k|^{(k-1)(|\lambda_1| + \ldots + |\lambda_k|)^2 + k-1}  $. 

Taking $X = \{0 , \ldots , |X| - 1\}$, one may check that $$ \lambda_1 \cdot X + \ldots + \lambda_k \cdot X \subset \{0 \ldots , (|\lambda_1| + \ldots + |\lambda_k|)(| X| - 1) \} + (|X| - 1) \sum_{\{i : \lambda_i < 0\}} \lambda_i ,$$ and so $$|\lambda_1 \cdot X + \ldots + \lambda_k \cdot X| \leq (|\lambda_1| + \ldots + |\lambda_k|) |X| - (|\lambda_1| + \ldots + |\lambda_k| - 1) .$$ This shows that Theorem \ref{main} is the best possible up to the additive constant. 

In \cite{Bu} it was shown that \begin{equation} |\lambda_1 \cdot A + \ldots + \lambda_k \cdot A| \geq (|\lambda_1| + \ldots + |\lambda_k|)|A| - o(|A|) \label{Bukh}, \end{equation} where $o(|A|)$ depends on $\lambda_1 , \ldots , \lambda_k$. The works of \cite{Ba, Ci, Du, Ha, Lj} made improvements from $o(|A|)$ to a in certain cases, all when $k = 2$. Indeed, in \cite{Ba}, the problem was completely resolved for $k=2$. For a more complete introduction of the problem, the reader is invited to see the introductions of \cite{Ba} and \cite{Bu}. Note that while Theorem \ref{main} is only claimed and proved for $A \subset \mathbb{Z}$, extending to $A \subset \mathbb{Q}$ is an obvious task by clearing denominators of $A$; and moreover, extending to $A \subset \mathbb{R}$ is handled by Lemma 5.25 in the book of Tao and Vu \cite{Ta}. 

Before we discuss the proof of Theorem \ref{main}, we mention two interesting problems related to the current topic. They are to estimate (from below) $|A + q \cdot A|$ where $q$ is an algebraic number and $A \subset \mathbb{Z}[q]$ and to estimate $|A + \pi \cdot A|$ where $\pi$ is any transcendental number and $A \subset \mathbb{R}$. 

Let $q$ be an algebraic number and $A \subset \mathbb{Z}[q]$ be finite. Finding good lower bounds for $|A+q \cdot A|$ is related to convex geometry. Let $d$ be the degree of the minimal polynomial of $q$ over $\mathbb{Q}$. Consider $L : \mathbb{Z}[q] \to \mathbb{Z}[q]$ via $x \mapsto qx$. Then $L$ is a $d$-dimensional linear transformation over $\mathbb{Q}$. The problem of finding bounds for $|A+q\cdot A|$ where $A \subset \mathbb{Z}[q]$ is equivalent to the problem of finding bounds for $|A + L A|$ where $A \subset \mathbb{Z}^d$. To this end, we are interested in lower bounds for $$|L_1  A + \ldots + L_k  A|,$$ where $L_i : \mathbb{Z}^d \to \mathbb{Z}^d$ are linear transformations for $1 \leq i \leq k$, $$L_1 \mathbb{Z}^d + \ldots +  L_k \mathbb{Z}^d = \mathbb{Z}^d,$$ and $L_1 , \ldots , L_k$ have no nontrivial common invariant subspace. These last two conditions generalize the necessary condition of coprimality for $d=1$.

Recall that the Brunn-Minkowski inequality states that if $X$ and $Y$ are nonempty bounded open subsets of $\mathbb{R}^d$, then $\mu(X+Y)^{1/d} \geq \mu(X)^{1/d} + \mu(Y)^{1/d}$ where $\mu$ is the Lebesgue measure on $\mathbb{R}^d$. Let $X$ be the open set formed by taking the union of all the unit boxes around each point of $A$. Note that $\mu(X) = |A|$. Now using that $L_1 \mathbb{Z}^d + \ldots  + L_k\mathbb{Z}^d = \mathbb{Z}^d$ one can observe that all of the integers points that intersect $L_1  X + \ldots + L_k  X$ and are not close to the perimeter lie in $L_1  A + \ldots + L_k  A$. Thus $\mu(L_1  X + \ldots + L_k  X)$ approximates $|L_1  A + \ldots + L_k  A|$ up to an error term that depends on the structure of $A$. This observation, coupled with our intuition that to minimize $|L_1  A + \ldots + L_k  A|$, one might expect $A$ to be a convex set intersected with $\mathbb{Z}^d$, we obtain the following question.

\begin{ques} Suppose $A$ is a finite subset of $\mathbb{Z}^d$ and $L_1 , \ldots , L_k: \mathbb{Z}^d \to \mathbb{Z}^d$ are linear transformations that have no common invariant subspace such that $L_1\mathbb{Z}^d + \ldots + L_k \mathbb{Z}^d = \mathbb{Z}^d$. Is it true that $$|L_1  A + \ldots + L_k  A| \geq (\det(L_1)^{1/d} + \ldots + \det(L_k)^{1/d})^d|A| - o(|A|)?$$ \end{ques} We expect the error term may be taken to be $O(|A|^{(d-1)/d})$, but any progress in this direction would be exciting to see. The author would like to thank Boris Bukh for showing him the above question as well as providing invaluable discussion in regards to Theorem \ref{main}.

The other problem is to estimate $|A+ \pi \cdot A|$ from below where $\pi$ is transcendental and $A \subset \mathbb{R}$. The growth is no longer linear in $|A|$, as it was shown in Corollary 3.7 of \cite{Ko} that $|A+ \pi \cdot A| \geq \Omega(|A| \frac{\log{|A|}}{\log{\log{|A|}}})$. The author guesses that the best one can do to minimize $|A+ \pi \cdot A|$ is to take a set of the form $$X = \{a_0 + a_1 \pi + a_2 \pi_2 + \ldots + a_n \pi^n : 0 \leq a_i < N_i,\ \ a_i \in \mathbb{Z} \},$$ where $n = \lfloor \sqrt{\log|X|} \rfloor$ and the $N_i$ differ by at most 1, giving $|X+\pi\cdot X| \leq 2^{O(\sqrt{\log|X|})}|X|$.

We would like to take a slight detour and discuss the analogous problem for finite fields of a prime order. The problem appears to be harder and much less is known. The only exact result in this direction is the celebrated Cauchy-Davenport theorem. Even the behavior of $|A + 2 \cdot A|$ when $A \subset \mathbb{F}_p$ is unclear. Both \cite{Pl} and \cite{Po} have partial results but the general question remains open. Constructions in \cite{Po} show that when $A \subset \mathbb{F}_p$ is large with respect to $p$, the size of $\lambda_1 \cdot A + \ldots + \lambda_k \cdot A$ can be smaller than what would be expected from the case where $A \subset \mathbb{Z}$. We remark that one can extend Theorem \ref{main} to very small subsets of $\mathbb{Z}_p$ by an argument similar to that of Theorem 1.3 in \cite{Po}.

\section{Proof of Theorem \ref{main}}

Let $\lambda_1 , \ldots , \lambda_k$ be coprime integers. Let $A$ be a finite subset of the integers. Throughout we assume Theorem \ref{main} has been shown for $k-1$ and we aim to show it holds for $\lambda_1 , \ldots , \lambda_k$. The case $k=1$ starts the induction, where we observe that the assumptions of Theorem \ref{main} imply $\lambda_1 = \pm 1$. We remark that we need the main ideas from \cite{Ba} to proceed with the proof of Theorem \ref{main}. Indeed, for $k=2$, the current proof and the proof from \cite{Ba} are nearly identical. 

Recall the well-known result that for any finite and nonempty subsets of the integers, $A$ and $B$, one has \begin{equation}\label{trivial} |A+B| \geq |A| + |B| - 1. \end{equation}
The following result generalizes this in a way which will prove to be quite useful.
\begin{lem}\label{GYM}[Gyarmati, Matolcsi, Ruzsa]\cite{Gy}
	Let $A_1 , \ldots , A_k$ be finite, nonempty subsets of the integers and denote \begin{align*} S_i :&= A_1 + \ldots + \widehat{A_{i}} + \ldots + A_k \\ 
		&= A_1 + \ldots + A_{i-1} + A_{i+1} + \ldots + A_k\end{align*} Then $$|A_1 + \ldots + A_k| \geq \frac{1}{k-1}\left( (\sum_{i=1}^k |S_i| ) - 1\right).$$
\end{lem}
It is worthwhile to check that Theorem \ref{main} for $k=2$, proved in \cite{Ba}, combined with Lemma \ref{GYM}, quickly implies Theorem \ref{main} when $\lambda_1 , \ldots , \lambda_k$ are pairwise coprime, but does not immediately handle the general case. We remark that the proof of Lemma \ref{GYM}, which is Theorem 1.4 in \cite{Gy}, is elegant and only occupies half of a page. 

For each $1 \leq i \leq k$, denote $$g_{i} := (\lambda_1 , \ldots , \lambda_{i-1} , \lambda_{i+1} , \ldots , \lambda_k).$$  Note we have that $(\lambda_i , g_{i}) = 1$ for each $1 \leq i \leq k$. This implies for each $1 \leq i < j \leq k$ that $(g_{i} , g_{j}) = 1$.

We establish some notation. Given $q \in \mathbb{Z}$, we say $A$ is {\em fully distributed} (FD) mod $q$ if $A$ intersects every residue class mod $q$. 

Fix $1 \leq i \leq k$. We now partition $A$ into its residue classes mod $g_{i}$, that is $$A = \displaystyle \bigcup_{j = 1}^{m_i} A_{ij}, \ \  A_{ij} = a_{ij} +  g_{i} \cdot A'_{ij}, \ \ A_{ij} \neq \emptyset ,\ \  0 \leq a_{ij} < g_{i},$$ where the union is disjoint. Then we have $$|A| = \displaystyle \sum_{j=1}^{m_i} |A_{ij}|.$$ 

Furthermore, we have \begin{equation}\label{sum}|\lambda_1 \cdot A + \ldots + \lambda_k \cdot A| = \sum_{j_1 = 1}^{m_1} \ldots \sum_{j_k = 1}^{m_k} |\lambda_1 \cdot A_{1j_1}  + \ldots + \lambda_k \cdot A_{k,j_k}|.\end{equation}

Since it does not cause much inconvenience, we commit ourselves to calculating the additive constant in Theorem \ref{main} in terms of the additive constants coming from the sums of the $k-1$ dilates. We remind the reader we are using induction on $k$ and that for $k=1$, we may take the additive constant to be zero. Let \begin{equation}\label{const1} C' =g_{1} C_{\lambda_2 \ldots , \lambda_{k}}  + \ldots +  g_{k} C_{\lambda_1 , \ldots , \lambda_{k-1}} , \end{equation} be a certain linear combination of the additive constants obtained from all of the possible $k-1$ dilates. 

\begin{lem}\label{FD}
	Suppose $A$ is FD mod $g_{i}$ for all $1 \leq i \leq k$. Then $$|\lambda_1 \cdot A + \ldots + \lambda_k \cdot A| \geq (|\lambda_1| + \ldots + |\lambda_k|) |A| - (C' + |\lambda_1 \cdots \lambda_k|).$$
\end{lem}
\begin{proof}
	Our assumption implies that for each $1 \leq i \leq k$, we have $m_i = g_{i}$.
	
	Set $$s_{\lambda_i} := |\lambda_1| + \ldots + |\lambda_{i-1}| + |\lambda_{i+1}| + \ldots + |\lambda_k|.$$
	
	Using \eqref{sum}, Lemma \ref{GYM}, and the induction hypothesis on $k$, we obtain 
	\begin{align*}
		|\lambda_1 \cdot A &+ \ldots + \lambda_k \cdot A|  = \sum_{j_1 =1}^{g_{1}} \cdots \sum_{j_k =1}^{g_{k}}|\lambda_1 \cdot A_{1j_1} + \ldots + \lambda_k \cdot A_{kj_k}| \\
		& \geq \sum_{j_1 =1}^{g_{1}} \cdots \sum_{j_k =1}^{g_{k}} \frac{1}{k-1} \left( (\sum_{i=1}^k |\lambda_1 \cdot A_{1j_1} + \ldots + \widehat{\lambda_i \cdot A_{ij_i}} + \ldots +\lambda_k \cdot A_{kj_k}| )- 1 \right) \\
		& \geq  \frac{1}{k-1} \sum_{i=1}^k g_{i} |\lambda_1 \cdot A + \ldots + \widehat{\lambda_i \cdot A} + \ldots +\lambda_k \cdot A| - \frac{1}{k-1} \prod_{i= 1}^k g_j \\
		& \geq  \frac{1}{k-1}\sum_{i=1}^k g_{i} \frac{s_{\lambda_i}}{g_{i}}|A| - \frac{1}{k-1} (C' + \prod_{i=1}^k g_{i}  ) \\
		& \geq (|\lambda_1| + \ldots + |\lambda_k|)|A|- \frac{1}{k-1} (C'+|\lambda_1 \cdots \lambda_k|) .
	\end{align*} 
\end{proof}


Lemma \ref{FD} proves Theorem \ref{main} in the case where $A$ is FD mod $g_{i}$ for all $1 \leq i \leq k$. 

Observe that translation and dilation of $A$ do not change $$|\lambda_1 \cdot A + \ldots + \lambda_k \cdot A|.$$ Every finite subset of the integers is equivalent, via an affine transformation, to a set which is not contained in any proper infinite arithmetic progression. We call such a set \emph{reduced} and we assume that $A$ is reduced. Thus we may and will assume that for all $1 \leq i \leq k$ and $1 \leq j \leq m_i$, we have that $$(a_{i1} - a_{ij} , a_{i2} - a_{ij}, \ldots , a_{im_i} - a_{ij} , g_{i}) = 1 .$$ Note that since $A$ is reduced, it must intersect at least two residue classes mod $g_{i}$, whenever $g_{i} >1$. We remark that more details about the reduction process can be found in \cite{Ba}. 

\begin{lem}\label{dist1}
	Fix $1 \leq i \leq k$ and then fix $1 \leq j \leq m_i$. Either $A'_{ij}$ is FD mod $g_{i}$ or \begin{align*}|\lambda_1 \cdot A + \ldots + \lambda_{i-1} \cdot A + \lambda_{i} & \cdot A_{ij} + \lambda_{i+1} \cdot A + \ldots + \lambda_k \cdot A| \\
		&\geq |\lambda_1 \cdot A_{ij} + \ldots + \lambda_k \cdot A_{ij}| + \min\limits_{1 \leq w \leq m_i} |A_{iw}|.\end{align*}
\end{lem}
\begin{proof}
	We show the result only for the case $i=1$; the rest follows by symmetry. Suppose $$|\lambda_1 \cdot A_{1j} + \lambda_2 \cdot A + \ldots + \lambda_k \cdot A| - |\lambda_1 \cdot A_{1j} + \ldots + \lambda_k \cdot A_{1j}| < \displaystyle \min_{1 \leq w \leq m_1} |A_{1w}|.$$ Fix $1 \leq h \leq m_1$. Then using that $A_{1h} , A_{1j} \subset A$, we obtain $$ Q :=\left|\left(\lambda_1 \cdot A_{1j} + \ldots + \lambda_{k-1} \cdot A_{1j} + \lambda_k \cdot A_{1h}\right) \setminus (\lambda_1 \cdot A_{1j} + \ldots + \lambda_k \cdot A_{1j})\right|  < \displaystyle \min_{1 \leq w \leq m_1} |A_{1w}| .$$  Translation by $(\lambda_1 + \ldots + \lambda_k)(-a_{1j})$ and then dilation by $\frac{1}{g_{1}}$ reveals that $$Q = \left|\left(\frac{\lambda_k}{g_{1}}(a_{1h} - a_{1j}) + \lambda_1 \cdot A'_{1j} + \ldots + \lambda_{k-1} \cdot A'_{1j} + \lambda_k \cdot A'_{1h}\right) \setminus \left(\lambda_1 \cdot A'_{1j} + \ldots + \lambda_k \cdot A'_{1j}\right)\right|.$$ Fix $x \in \lambda_1 \cdot A'_{1j}$. Since $$Q < \displaystyle \min_{1 \leq w \leq m_i} |A'_{1w}| \leq |A'_{1h}| \leq | \lambda_2\cdot A_{1j}+\ldots +  \lambda_{k-1} \cdot A_{1j}+\lambda_k\cdot A_{1h}|,$$ we see that there exists a $$y \in \lambda_2 \cdot A'_{1j} + \ldots + \lambda_{k-1} \cdot A'_{1j} + \lambda_k \cdot A'_{1h},$$ such that $$\frac{\lambda_k}{g_{1}}(a_{1h} - a_{1j}) +  x + y \in \lambda_1 \cdot A'_{1j} + \ldots + \lambda_k \cdot A'_{1j}.$$ Thus there is a $x_2 \in \lambda_1 \cdot  A'_{1j}$ such that $\frac{\lambda_k}{g_{1}}(a_{1h} - a_{1j}) + x \equiv  x_2 \mod g_{1}$. We may repeat this argument with $x_2$ in the place of $x$, and so on, to obtain for any $v \in \mathbb{Z}$ there is a $x_3 \in \lambda_1 \cdot  A'_{1j}$ such that $v \frac{\lambda_k}{g_{1}}(a_{1h} - a_{1j}) + x \equiv x_3 \mod g_{1}$. We may also repeat this argument with $x_3$ in the place of $x$ and with $\lambda_c$ in the place of $\lambda_k$ for each $2 \leq c \leq k$. Since $(\frac{\lambda_2}{g_{1}} , \ldots , \frac{\lambda_k}{g_{1}}) = 1$, we have that there are integers $v_2 , \ldots , v_k$ such that $\sum_{c = 2}^k v_{c} \frac{\lambda_c}{g_{1}} = 1$. From this, we may infer there is an $x_4 \in \lambda_1 \cdot A'_{1j}$ such that $$(a_{1h} - a_{1j}) + x  = \sum_{c = 2}^k v_{c} \frac{\lambda_c}{g_{1}} (a_{1h} - a_{1j}) + x \equiv x_4 \mod g_{1}.$$ Now we may repeat this argument with $x_4$ in the place of $x$, and so on, and for each $1 \leq h \leq m_1$. We conclude for each $u_1 , \ldots , u_{m_1} \in \mathbb{Z}$, there is a $x_5 \in \lambda_1 \cdot A'_{1j}$ such that $$u_1(a_{11} - a_{1j}) + \ldots + u_{m_1}(a_{1m_1} - a_{1j}) + x \equiv x_5 \mod g_{1}.$$ Since $A$ is reduced, the set of possible values of $x_5$ meets every residue class mod $g_{1}$, so we find that $\lambda_1 \cdot A'_{1j}$ is FD mod $g_{1}$. Since $(\lambda_1 , g_{1}) = 1$, it follows that $A'_{1j}$ is FD mod $g_{1}$. 
\end{proof}

Let $p = \prod_{i=1}^k g_{i}$. We partition $A$ into its residue classes mod $p$, that is $$A = \displaystyle \bigcup_{e = 1}^{m} P_e, \ \  P_e = p_e +  p \cdot P'_e, \ \ P_e \neq \emptyset ,\ \  0 \leq p_e < p,$$ where the union is disjoint. Then we have $$|A| = \displaystyle \sum_{e=1}^{m} |P_e|.$$ 

For a fixed $1 \leq e \leq m$ and $1 \leq i \leq k$, let $e_i \in \{1 , \ldots , m_i\}$ such that $p_e \equiv a_{ie_i}$ mod $ g_{i}$. Clearly $P_e \subset A_{ie_i}$, and actually $$P_e=\bigcap_{i=1}^k A_{ie_i}.$$

\begin{lem}\label{dist2}
	Let $1 \leq e \leq m$. Suppose for all $1 \leq i \leq k$ that $A'_{ie_{i}}$ is FD mod $g_{i}$. Then either $P'_e$ is FD mod $g_{i}$ for all $1 \leq i \leq k$ or $$|\lambda_1 \cdot A_{1 e_{1}} + \ldots + \lambda_k \cdot A_{ke_k}| \geq |\lambda_1 \cdot P_e + \ldots + \lambda_k \cdot P_e| + |P_e|.$$
\end{lem}
\begin{proof}
	Suppose $$|\lambda_1 \cdot A_{1 e_{1}} + \ldots + \lambda_k \cdot A_{k,e_k}| -|\lambda_1 \cdot P_e + \ldots + \lambda_k \cdot P_e| < |P_e|.$$ We show that $P'_e$ is FD mod $g_{1}$ and the rest will follow by symmetry. Using that $P_e \subset A_{1e_2} , \ldots , A_{1e_k}$, we find that $$Q :=|(\lambda_1 \cdot A_{1 e_{1}} + \lambda_2 \cdot P_e + \ldots + \lambda_k \cdot P_e) \setminus (\lambda_1 \cdot P_e + \ldots + \lambda_k \cdot P_e)| < |P_e|.$$ Let $d := p/g_{1}$. Translation by $(\lambda_1 + \ldots + \lambda_k)(-p_e)$ and then dilation by $\frac{1}{g_{1}}$ reveals $$Q = |\lambda_1 \frac{(a_{1e_1} - p_e)}{g_{1}} + \lambda_1 \cdot A'_{1 e_{1}} + d\lambda_2 \cdot P'_e + \ldots + d\lambda_k \cdot P'_e) \setminus (d\lambda_1 \cdot P'_e + \ldots + d\lambda_k \cdot P'_e)| .$$ Observe that $\lambda_1 \frac{(a_{1e_1} - p_e)}{g_{1}}$ is an integer. Fix $a' \in A'_{1e_1}$. Since $Q < |P'_e|$, there is a $y \in d\lambda_2 \cdot P'_e + \ldots + d\lambda_k \cdot P'_e$ such that $$\lambda_1 \frac{(a_{1e_1} - p_e)}{g_{1}} + \lambda_1 a' + y \in d\lambda_1 \cdot P'_e + \ldots + d\lambda_k \cdot P'_e.$$ Thus we may find an $p' \in P'_e$ such that $\lambda_1 \frac{(a_{1e_1} - p_e)}{g_{1}} +\lambda_1 a' \equiv d \lambda_1 p' \mod g_{1}$. Since $(\lambda_1 , g_{1}) =1$, this set of $\lambda_1 \frac{(a_{1e_1} - p_e)}{g_{1}} +\lambda_1 a'$ meets every residue class mod $g_{1}$. It follows that $d \lambda_1 \cdot P'_e$ is FD mod $g_{1}$. But $(d\lambda_1 , g_{1}) = 1$ and so $P'_e$ is FD mod $g_{1}$.

\end{proof}


We informally sketch the argument of our next proposition, Proposition \ref{mainprop}, which will immediately imply Theorem \ref{main}. We shall prove a sequence of estimates of the form $|\lambda_1 \cdot A+...+\lambda_k\cdot A|\geq M|A|-C$ for various values of $M$ and $C$. We start with the trivial inequality $| \lambda_1 \cdot A+...+\lambda_k\cdot A| \geq |A|$, and then successively improve the multiplicative constant $M$, only asking to increase it by $\frac{1}{|\lambda_1| + \ldots + |\lambda_k|}$ in each step. The proof breaks into four cases:

\begin{itemize}
	\item[(i)]\ \  If some $A_{ij}$ is small, then we use \eqref{trivial},
	\item[(ii)] \ \  If some $A_{ij}$ is not FD mod $g_{i}$, then use Lemma \ref{dist1} and $A_{ij}$ is big from $(i)$,
	\item[(iii)] \ \ If $P'_e$ is not FD mod $g_{i}$, then use Lemma \ref{dist2} and $A_{ie_i}$ is FD mod $g_{i}$ from $(ii)$,
	\item[(iv)] \ \  If $P'_e$ is FD mod $g_{i}$ for all $1 \leq i \leq k$ and use Lemma \ref{FD}.
\end{itemize}
The first three cases use the previous estimates of $|\lambda_1 \cdot A + \ldots + \lambda_k \cdot A| \geq M|A|-C$ for various choices of $A$. Note that $(iv)$ is what prevents us from improving upon the multiplicative constant beyond $|\lambda_1| + \ldots + |\lambda_k|$.

Let \begin{equation}\label{const2} C'' = \frac{1}{k-1} (|\lambda_1 \cdots \lambda_k|)( C_{\lambda_2 \ldots , \lambda_{k}}  + \ldots +  C_{\lambda_1 , \ldots , \lambda_{k-1}} + 1) \end{equation} which is a modest upper bound for $\frac{1}{k-1}(C' + |\lambda_1 \cdots \lambda_k|)$ where $C'$ is defined in \eqref{const1}. 

\begin{prop}\label{mainprop}
	Let $\lambda_1 , \ldots , \lambda_k \in \mathbb{Z}$ be coprime, $A \subset \mathbb{Z}$ be finite. Then for all $|\lambda_1| + \ldots+  |\lambda_k| \leq u \leq (|\lambda_1| + \ldots + |\lambda_k|)^2$, $$|\lambda_1 \cdot A + \ldots + \lambda_k \cdot A| \geq \frac{u}{|\lambda_1| + \ldots +  |\lambda_k|}|A| - C''(|\lambda_1 |\ldots|\lambda_k|)^u ,$$ where $C''$ is defined in \eqref{const2}. 
\end{prop}

\begin{proof}
	
	For convenience, let $S = |\lambda_1| + \ldots+ |\lambda_k|$ and $C_u =  (|\lambda_1 |\ldots|\lambda_k|)^u$. We induct on $u$, where the inequality $|\lambda_1 \cdot A + \ldots + \lambda_k \cdot A| \geq |A|$ starts the induction for $u = S$. 
	
	First, assume there is a $1 \leq j \leq m_1$ such that $|A_{1j}| \leq \frac{1}{S} |A|$. Then using the induction hypothesis and \eqref{trivial} we obtain
	\begin{align*}
		|\lambda_1& \cdot A + \ldots+ \lambda_k \cdot A| \\
		&\geq |\lambda_1 \cdot A_{1j} + \lambda_2 \cdot A+ \ldots + \lambda_k \cdot A| + |\lambda_1 \cdot (A \setminus A_{1j}) + \ldots+ \lambda_k \cdot (A \setminus A_{1j})| \\
		&\geq |A_{1j}| + |A| - 1 +  \frac{u}{S} (|A| - |A_{1j}|) - C_u \\
		& \geq \frac{u+1}{S} |A| - C_{u+1},
	\end{align*}
	using in the last step that $u < S^2$ and $ C_u + 1 \leq C_{u+1}$. By the symmetry of $\lambda_1 , \ldots , \lambda_k$, we may assume, for every $1 \leq i \leq k$ and $1 \leq j \leq m_i$, that $A_{ij}$ has more than $\frac{1}{S}|A|$ elements. 
	
	Suppose there is some $1 \leq j \leq m_1$ such that $A'_{1j}$ that is not FD mod $g_{1}$. Then by Lemma \ref{dist1} and the induction hypothesis
	
	\begin{align*}
		&|\lambda_1 \cdot A + \ldots + \lambda_k \cdot A| \\
		&\geq |\lambda_1 \cdot A_{1j} + \lambda_2 \cdot A + \ldots + \lambda_k \cdot A| + |\lambda_1 \cdot (A \setminus A_{1j}) + \ldots+ \lambda_k \cdot (A \setminus A_{1j})| \\
		& \geq |\lambda_1 \cdot A_{1j} + \ldots  + \lambda_k \cdot A_{1j}| + \min_{1 \leq w \leq m_1} |A_{1w}| + \frac{u}{S}(|A| - |A_{1j}|) - C_u \\
		& \geq \frac{u}{S} |A_{1j}| - C_u + \frac{1}{S} |A|  + \frac{u}{S}(|A| - |A_{1j}|) - C_u \\
		& \geq \frac{u+1}{S} |A| - C_{u+1},
	\end{align*}
	using $C_{u+1} \geq 2 C_u$ in the last step. By the symmetry of $\lambda_1 , \ldots , \lambda_k$, we may now assume, for each $1 \leq i \leq k$ and for each $1 \leq j \leq m_i$, that $A'_{ij}$ is FD mod $g_{i}$. 
	
	Fix $1 \leq e \leq m$. Then using Lemma \ref{dist2}, we obtain either \begin{align*} |\lambda_1 \cdot A_{1e_1} + \ldots +  \lambda_k \cdot A_{ke_k}| &\geq |\lambda_1 \cdot P_e + \ldots+ \lambda_k \cdot P_e| + |P_e| \\ 
		& \geq \frac{u}{S} |P_e| - C_u + |P_e| \\
		&\geq \frac{u+1}{S} |P_e| - C_u,\end{align*} by the induction hypothesis or that $P'_e$ is FD mod $g_{i}$ for all $1 \leq i \leq k$. In the latter case, by Lemma \ref{FD}, using that $u < S^2$, we have $$|\lambda_1 \cdot P_e + \ldots + \lambda_k \cdot P_e| = |\lambda_1 \cdot P'_e + \ldots + \lambda_k \cdot P'_e| \geq S|P'_e| - C'' \geq \frac{u+1}{S} |P_e| - C_u,$$ where $C''$ is defined in \eqref{const2}. 
	In either case we obtain $$|\lambda_1 \cdot A_{1e_1}+ \ldots + \lambda_k \cdot A_{ke_k}| \geq \frac{u+1}{S} |P_e| - C_u.$$ Using \eqref{sum}, we find
	
	\begin{align*}
		|\lambda_1 \cdot A + \ldots + \lambda_k \cdot A| &= \sum_{j_1 = 1}^{m_1} \ldots \sum_{j_k = 1}^{m_k} |\lambda_1 \cdot A_{1j_1}  + \ldots + \lambda_k \cdot A_{k,j_k}|\\
		& \geq \sum_{e=1}^m \left(\frac{u+1}{S} |P_e| - C_u\right) \geq \frac{u+1}{S} |A| - C_{u+1}
	\end{align*}
	since $C_{u+1} \geq |\lambda_1 \cdots \lambda_k| C_u \geq  mC_u$.
\end{proof}

Allowing $u = (|\lambda_1| + \ldots + |\lambda_k|)^2$ in Proposition \ref{mainprop} yields Theorem \ref{main}, where we may take 
\begin{align*} C &= \frac{1}{k-1}(C_{\lambda_2 , \ldots , \lambda_k} + \ldots + C_{\lambda_1 + \ldots + \lambda_{k-1}} + 1)|\lambda_1 \cdots \lambda_k|^{(|\lambda_1 \cdots \lambda_k|)^2 + 1} \\
	&\leq \frac{1}{3} {k+1 \choose 2} |\lambda_1 \cdots \lambda_k|^{(k-1)(|\lambda_1| + \ldots + |\lambda_k|)^2 + k-1} .\end{align*}

The author would like to thank Antal Balog for not only providing useful suggestions for this paper, but introducing him to the current problem, as well as the beautiful subject of additive combinatorics. Also, the author would like to thank the referee for their useful suggestions. 

Written while the author enjoyed the hospitality of the Alfr\'{e}d R\'{e}nyi 
Institute of Mathematics, and benefited from the OTKA grant K 109789.

\end{document}